\documentclass[final,3p,times]{elsarticle}

\usepackage{hyperref}
\usepackage{mathtools, amsfonts, amssymb, mathrsfs, amsopn}
\usepackage{enumerate,enumitem}
\usepackage{subcaption}
\usepackage{soul}
\usepackage{comment}
\usepackage{xspace}
\usepackage{mleftright}
\usepackage{booktabs}
\usepackage{algorithm, algpseudocode, algorithmicx}

\usepackage{amsthm}
\newtheorem{theorem}{Theorem}
\newtheorem{lemma}{Lemma}
\newtheorem{proposition}{Proposition}
\newtheorem{corollary}{Corollary}

\usepackage{xcolor}
\definecolor{badcolor}{HTML}{FF5733}
\definecolor{goodcolor}{HTML}{2196F3}
\definecolor{greatcolor}{HTML}{0072B2}

\newcommand{\real}{\mathbb{R}}

\DeclareMathOperator{\diag}{diag}

\newcommand{\mat}[1]{\boldsymbol{#1}}
\renewcommand{\vec}[1]{\boldsymbol{#1}}

\newcommand{\norm}[1]{\mleft\| #1 \mright\|}

\newcommand{\Id}{\mathbf{I}}

\DeclareMathOperator{\expect}{\mathbb{E}}
\DeclareMathOperator{\prob}{\mathbb{P}}

\DeclareMathOperator*{\argmin}{argmin}
\newcommand{\set}[1]{\mathsf{#1}}

\renewcommand{\hat}[1]{\widehat{#1}}
\renewcommand{\tilde}[1]{\widetilde{#1}}

\usepackage{pifont}
\newcommand{\cmark}{\ding{51}}%
\newcommand{\xmark}{\ding{55}}%

\journal{Nuclear Physics B}

\begin{document}

\begin{frontmatter}



\title{Randomized Kaczmarz with tail averaging}


\author[label1]{Ethan N. Epperly} 
\ead{eepperly@caltech.edu}

\affiliation[label1]{organization={Department of Computing and Mathematical Sciences, California Institute of Technology},
            addressline={1200 E California Blvd}, 
            city={Pasadena},
            postcode={91125}, 
            state={CA},
            country={USA}}

\author[label2]{Gil Goldshlager} 
\ead{ggoldsh@berkeley.edu}

\affiliation[label2]{organization={Department of Mathematics, University of California Berkeley},
            addressline={110 Sproul Hall}, 
            city={Berkeley},
            postcode={94720}, 
            state={CA},
            country={USA}}

\author[label3]{Robert J. Webber} 
\ead{rwebber@ucsd.edu}

\affiliation[label3]{organization={Department of Mathematics, University of California San Diego},
            addressline={9500 Gilman Drive}, 
            city={La Jolla},
            postcode={92093}, 
            state={CA},
            country={USA}}

\begin{abstract}
The randomized Kaczmarz (RK) method is a well-known approach for solving linear least-squares problems with a large number of rows.
RK accesses and processes just one row at a time, leading to exponentially fast convergence for consistent linear systems.
However, RK fails to converge to the least-squares solution for inconsistent systems.
This work presents a simple fix: average the RK iterates produced in the tail part of the algorithm.
The proposed tail-averaged randomized Kaczmarz (TARK) converges for both consistent and inconsistent least-squares problems at a polynomial rate, which is known to be optimal for any row-access method.
An extension of TARK also leads to efficient solutions for ridge-regularized least-squares problems.
\end{abstract}







\end{frontmatter}



\section{Introduction}

The overdetermined linear least-squares problem
\begin{equation} \label{eq:ls}
   \min_{\vec{x} \in \real^d}\, \norm{\vec{b}  - \mat{A}\vec{x}}^2 \quad \text{for } \mat{A} \in \real^{n\times d} \text{ and } \vec{b} \in \real^n \text{ with } n > d
\end{equation}
is fundamental in statistics, scientific computation, and machine learning.
Its solution is conveniently expressed using the Moore--Penrose pseudoinverse, $\vec{x}_\star = \mat{A}^+ \mat{b}$.
However, computing this solution by direct means is slow and memory-intensive when the number of rows is large.
For the largest problems
(say, $n \geq 10^{12}$), storing even a single column of $\mat{A}$ in random-access memory is challenging.

\emph{Row-access} methods have been proposed as a practical way to solve large least-squares problems.
These methods access and process one or a few rows of $\mat{A}$ at a time.
An example of a row-access method is randomized Kaczmarz (RK) \cite{strohmer2008randomized}, which is reviewed in Subsection~\ref{sec:kaczmarz}.
RK converges exponentially fast if the least-squares problem is consistent, $\vec{b} = \mat{A}\vec{x}_\star $ \cite{strohmer2008randomized,ma2015convergence}.
However, in the inconsistent case where $ \vec{b} \ne \mat{A}\vec{x}_\star$, RK only converges up to a finite horizon. 
This paper overcomes the finite horizon by combining RK with tail averaging, resulting in a new tail-averaged randomized Kaczmarz (TARK) method.

\subsection{Randomized Kaczmarz} \label{sec:kaczmarz}
Randomized Kaczmarz \cite{strohmer2008randomized} is a well-known row-access method.
Beginning with an initial estimate (typically $\vec{x}_0=\vec{0}$), RK applies the following update procedure for $t=0,1,\ldots$:
\begin{itemize}
    \item \emph{\textbf{Sample}} a row index $i_t$ according to the probability distribution 
    \begin{subequations} \label{eq:rk}
    \begin{equation}
        \prob \{ i_t = i \} = \frac{\norm{\vec{a}_i}^2}{\norm{\mat{A}}_{\rm F}^2} \quad \text{for } i =1, \ldots, n. \label{eq:rk-sample}
    \end{equation}
    \item \emph{\textbf{Update}} the solution $\vec{x}_t$ so that the selected equation $\vec{a}_{i_t}^\top \vec{x} = b_{i_t}$ holds exactly:
    \begin{equation}
        \vec{x}_{t+1} \coloneqq \vec{x}_t + \frac{b_{i_t} - \vec{a}_{i_t}^\top \vec{x}_t}{\norm{\vec{a}_{i_t}}^2} \vec{a}_{i_t}. \label{eq:rk-update}
    \end{equation}
    \end{subequations}
\end{itemize}
Throughout this paper, $\vec{a}_i^\top$ denotes the $i$th row of $\mat{A}$, $b_i$ denotes the $i$th entry of $\vec{b}$, $\norm{\cdot}$ is the vector $\ell_2$ norm or matrix spectral norm, and $\norm{\cdot}_{\rm F}$ is the Frobenius norm.

RK can be interpreted as an optimized version of stochastic gradient descent for linear least-squares problems
that uses nonuniform selection probabilities to improve the convergence rate and eliminate the need for step size tuning \cite{needell2014stochastic}.
These probabilities can be precomputed using a single pass through the matrix $\mat{A}$, which might be expensive.
Sometimes this initial computation can be avoided by using rejection sampling \cite[Sec.~3]{needell2014stochastic}. Alternatively, RK can be implemented with uniform sampling, which is equivalent to applying RK to the diagonally reweighted least-squares problem $\min_{\vec{x}}\, \norm{\mat{D}\vec{b} - (\mat{D}\mat{A})\vec{x}}^2$ for $\mat{D} = \diag(1/\norm{\vec{a}_i})$.

The convergence rate for RK depends on the Demmel condition number 
\begin{equation*} \label{eq:demmel}
    \kappa_{\rm dem} \coloneqq \norm{\smash{\mat{A}^+}} \,\norm{\mat{A}}_{\rm F}.
\end{equation*}
The best available error bound is as follows:
\begin{theorem}[Randomized Kaczmarz: Convergence to a horizon \cite{zouzias2013randomized}] \label{thm:horizon}
    Assume $\vec{x}_0 \in \operatorname{range}(\mat{A}^\top)$.
    Then the RK iteration \eqref{eq:rk} converges exponentially fast until reaching a finite horizon related to the inconsistency:
    \begin{equation*}
        \expect\, \norm{\vec{x}_t - \vec{x}_\star}^2 \le \underbrace{\bigl(1 - \kappa_{\rm dem}^{-2} \bigr)^t\, \norm{\vec{x}_0 - \vec{x}_\star}^2}_{\text{exponential convergence}} + \underbrace{\norm{\smash{\mat{A}^+}}^2\, \lVert \vec{b} - \mat{A} \vec{x_\star} \rVert^2}_{\text{finite horizon}}.
    \end{equation*}
\end{theorem}

Unfortunately, the finite convergence horizon cannot be eliminated without changing the RK algorithm.
To overcome this obstacle, several variants of RK have been proposed:
\begin{itemize}
\item Randomized extended Kaczmarz \cite{zouzias2013randomized} manipulates the columns of $\mat{A}$ to achieve exponential convergence to $\vec{x}_\star$, even in the inconsistent case.
Yet the column manipulations are prohibitively expensive for the largest problems.
\item RK with underrelaxation (RKU) \cite{censor1983strong,cai2012exponential} introduces a relaxation parameter that can be gradually reduced to ensure convergence to the least-squares solution $\vec{x}_\star$.
The available theory suggests the method no longer converges exponentially fast for consistent problems \cite{moulines2011non,lin2015learning}.
\item Randomized Kaczmarz with averaging (RKA) \cite{moorman2020randomized} averages multiple independent RK updates (``threads'') at each iteration.
This method still converges only up to a finite horizon, but the horizon can be reduced by increasing the number of threads.
\end{itemize}
The limitations of these existing methods will be demonstrated through the experiments in Subsection~\ref{sec:numerics}.

\subsection{Tail-averaged randomized Kaczmarz}
\label{sec:tail}

This paper explores tail averaging as a different strategy to improve the convergence of RK.
Given a sequence of iterates $\vec{x}_0, \vec{x}_1, \ldots$,
the tail-averaged estimator is the quantity
\begin{equation}
\label{eq:averaged}
    \overline{\vec{x}}_t \coloneqq \frac{1}{t-t_{\rm b}} \sum\nolimits_{s = t_{\rm b}}^{t-1} \vec{x}_s,
\end{equation}
which depends on the burn-in time $t_{\rm b}$ and the final time $t$.
Tail averaging is frequently applied in Markov chain Monte Carlo \cite{metropolis1953equation} to obtain a convergent estimator from stochastically varying samples.
Tail averaging has also been combined with numerical optimization methods \cite[Thm.~3.2]{bubeck2015convex}, and it leads to the optimal $\mathcal{O}(1/t)$ convergence rate for stochastic gradient descent for strongly convex loss functions \cite{polyak1992acceleration,rakhlin2012making,jain2018parallelizing}.

\begin{algorithm}[t]
\caption{Tail-averaged randomized Kaczmarz (TARK)} \label{alg:TARK}
\begin{algorithmic}[1]
\Require Matrix $\mat{A} \in \real^{n \times d}$, vector $\vec{b} \in \real^n$, initial estimate $\vec{x}_0 \in \real^d$, burn-in time $t_{\rm b}$, and final time $t$
\For{$s$ in $0, \ldots, t-2$}
    \State Sample $i \sim \norm{\vec{a}_i}^2 /\norm{\mat{A}}_{\rm F}^2$ 
    \State $\vec{x}_{s+1}= \vec{x}_s + (b_i - \vec{a}_i^\top \vec{x}_s) \,\vec{a}_i / \norm{\vec{a}_i}^2$
    \EndFor 
\State $\overline{\vec{x}}_t = \bigl(\sum_{s = t_{\rm b}}^{t-1} \vec{x}_s\bigr) / (t - t_b)$ \\
\Return $\overline{\vec{x}}_t$ 
\end{algorithmic}
\end{algorithm}

Our main proposal is \emph{tail-averaged} randomized Kaczmarz (TARK), which outputs the tail average \eqref{eq:averaged} of the standard RK iterates \eqref{eq:rk};
see Algorithm~\ref{alg:TARK}.
A variant of TARK for ridge regression problems will be presented in Subsection~\ref{sec:rr}.

TARK converges to the exact least-squares solution $\vec{x}_\star$ with no finite horizon, for both consistent and inconsistent least-squares problems: 
\begin{theorem}[Mean square error bound for TARK]
\label{thm:variance_simple}
Assume $\vec{x}_0 \in \operatorname{range}(\mat{A}^\top)$.
The TARK estimator converges at a hybrid rate that balances exponential and polynomial convergence:
\begin{equation*}
    \expect\, \bigl\lVert \overline{\vec{x}}_t - \vec{x}_\star \bigr\rVert^2 \leq 
    \underbrace{\bigl(1 - \kappa_{\rm dem}^{-2} \bigr)^{t_{\rm b}} \, \norm{\vec{x}_0 - \vec{x}_\star}^2}_{\text{exponential convergence}}
    + \underbrace{\frac{2 \kappa_{\rm dem}^2 - 1}{t - t_{\rm b}} \, \norm{\smash{\mat{A}^+}}^2 \, \lVert \vec{b} - \mat{A} \vec{x_\star} \rVert^2}_{\text{polynomial convergence}}.
\end{equation*}
\end{theorem}
The proof of Theorem~\ref{thm:variance_simple} appears in Subsection~\ref{sec:proof}.

Similar to MCMC error bounds,
Theorem~\ref{thm:variance_simple} decomposes the mean square error into the sum of a bias term that decays exponentially in the burn-in time $t_{\rm b}$
and a variance term that decays as $1/t$ in the final time $t$. In particular, TARK converges when both $t_{\rm b}$ and $t-t_{\rm b}$ go to infinity.
To control both terms in this error bound, we recommend selecting $t_{\rm b} \in [t/4, t/2]$; see \ref{app:increase} for a storage-efficient implementation that ensures this condition when the final time $t$ is not known in advance.

\begin{table}[t]
    \centering
    \begin{tabular}{lcccccl}\toprule
    Method & Initial rate & Final rate & Row-access \\ \midrule 
    RK & \textcolor{goodcolor}{Exponential} & \textcolor{badcolor}{Finite horizon} & \textcolor{goodcolor}{Yes \cmark} \\ \midrule
    Extended RK \cite{zouzias2013randomized} & \textcolor{goodcolor}{Exponential} & \textcolor{goodcolor}{Exponential} & \textcolor{badcolor}{No \xmark} \\
    RK w/ underrelaxation \cite{cai2012exponential} & \textcolor{badcolor}{Less than exponential} & \textcolor{goodcolor}{Polynomial} & \textcolor{goodcolor}{Yes \cmark} \\
    RK w/ averaging \cite{moorman2020randomized} & \textcolor{goodcolor}{Exponential} & \textcolor{badcolor}{Finite horizon} & \textcolor{goodcolor}{Yes \cmark} \\ \midrule 
    TARK & \textcolor{goodcolor}{Exponential} & \textcolor{goodcolor}{Polynomial} & \textcolor{goodcolor}{Yes \cmark}
    \\ \bottomrule 
    \end{tabular}
    \caption{RK variants for inconsistent least-squares problems. The table lists the initial rate of convergence, the final rate of convergence, and whether the method is a row-access method.}
    \label{tab:comparison}
\end{table}

A similar proof guarantees that TARK converges when $t$ goes to infinity with $t_{\rm b}$ fixed. 
We have the following alternative version of Theorem~\ref{thm:variance_simple}:
\begin{theorem}[Alternative TARK error bound]
\label{thm:alternative}
Assume $\vec{x}_0 \in \operatorname{range}(\mat{A}^\top)$.
The TARK estimator satisfies the alternative error bound:
\begin{equation*}
    \expect\, \bigl\lVert \overline{\vec{x}}_t - \vec{x}_\star \bigr\rVert^2 \leq 
    \frac{2 \kappa_{\rm dem}^2 - 1}{t - t_{\rm b}}
    \biggl[
    \frac{\kappa_{\rm dem}^2 \left(1 - \kappa_{\rm dem}^{-2} \right)^{t_{\rm b}}}{(t - t_{\rm b})} \, \norm{\vec{x}_0 - \vec{x}_\star}^2
    +  \norm{\smash{\mat{A}^+}}^2\, \lVert \vec{b} - \mat{A} \vec{x_\star} \rVert^2 \biggr].
\end{equation*}
\end{theorem}
The proof of Theorem~\ref{thm:alternative} appears in Subsection~\ref{sec:proof}.

Based on our literature survey and discussions with RK experts, we believe that TARK is new.
Table~\ref{tab:comparison} presents a comparison of TARK with previous RK variants.

\section{Analysis and evaluation of tail-averaged randomized Kaczmarz} \label{sec:tark} 
This section provides a more detailed discussion of TARK.
Subsection~\ref{sec:proof} proves Theorems~\ref{thm:variance_simple} and \ref{thm:alternative}, Subsection~\ref{sec:optimality} discusses the optimal convergence rate for row-access methods, Subsection~\ref{sec:numerics} provides numerical experiments, and Subsection~\ref{sec:semi} extends TARK to semi-infinite least-squares problems.

\subsection{Proof of main theorem} \label{sec:proof}
The proof of Theorem~\ref{thm:variance_simple} follows the pattern of analysis initiated in \cite{strohmer2008randomized}, but it takes a step further by bounding the inner product terms $\expect \bigl[(\vec{x}_{t+s} - \vec{x}_\star)^\top (\vec{x}_t - \vec{x}_\star) \bigr]$ which decay exponentially fast with $s$.
For ease of reading, the analysis is presented as three lemmas followed by one main calculation.

\begin{lemma}[Multi-step expectations]
\label{lem:expectations}
    The RK iteration \eqref{eq:rk} satisfies
    \begin{equation*}
        \expect\, \bigl[\vec{x}_s - \vec{x}_\star\, \big|\, \vec{x}_r \bigr]
        = \Biggl[\mathbf{I} - \frac{\mat{A}^\top \mat{A}}{\norm{\mat{A}}_{\rm F}^2}\Biggr]^{s-r} \bigl( \vec{x}_r - \vec{x}_\star \bigr),
    \end{equation*}
    for any $r < s$,
    where the expectation averages over the random indices $i_r, \ldots, i_{s-1}$.
\end{lemma}
\begin{proof}
    For any $t \in \{r, \ldots, s-1\}$, write the one-step update \eqref{eq:rk-update} as
    \begin{equation*}
    \vec{x}_{t+1} - \vec{x}_\star
    = \vec{x}_t + \frac{b_{i_t} - \vec{a}_{i_t}^\top \vec{x}_t}{\norm{\vec{a}_{i_t}}^2} \vec{a}_{i_t} - \vec{x}_\star
    = \biggl[\mathbf{I} - \frac{\vec{a}_{i_t}^{\vphantom{\top}} \vec{a}_{i_t}^\top}{\lVert \vec{a}_{i_t} \rVert^2}\biggr] \bigl(\vec{x}_t - \vec{x}_\star\bigr)
    + \frac{b_{i_t} - \vec{a}_{i_t}^\top \vec{x}_\star}{\norm{\vec{a}_{i_t}}^2} \vec{a}_{i_t}.
    \end{equation*}
    Use the sampling probabilities \eqref{eq:rk-sample} to calculate the expectation over the random index $i_t$:
    \begin{equation*}
        \expect\, \bigl[ \vec{x}_{t+1} - \vec{x}_\star \, \big| \, \vec{x}_t \bigr]
        = \Biggl[\mathbf{I} - \frac{\mat{A}^\top \mat{A}}{\norm{\mat{A}}_{\rm F}^2}\Biggr]\bigl( \vec{x}_t - \vec{x}_\star \bigr)
        + \frac{\vec{A}^\top \bigl(\vec{b} - \mat{A} \vec{x}_\star\bigr)}{\norm{\mat{A}}_{\rm F}^2}.
    \end{equation*}
    The least-squares solution $\vec{x}_\star $ satisfies the normal equations $\vec{A}^\top \bigl(\vec{b} - \mat{A} \vec{x}_\star\bigr) = \vec{0}$, so the last term vanishes.
    Next, take the expectation over the random indices $i_r, \ldots, i_t$:
    \begin{equation*}
        \expect\, \bigl[ \vec{x}_{t+1} - \vec{x}_\star \,\big|\, \vec{x}_r \bigr]
        = \Biggl[\mathbf{I} - \frac{\mat{A}^\top \mat{A}}{\norm{\mat{A}}_{\rm F}^2}\Biggr] \expect\, \bigl[ \vec{x}_t - \vec{x}_\star \, \big|\, \vec{x}_r \bigr],
        \quad \text{for each } t \in \{r, \ldots, s-1\},
    \end{equation*}
    Iterating this equation completes the proof.
\end{proof}

\begin{lemma}[Demmel condition number bound]
\label{lem:demmel}
    Assume $\vec{x}_0 \in \operatorname{range}(\mat{A}^\top)$.
    Then the RK iteration \eqref{eq:rk} satisfies
    \begin{equation*}
        \bigl(\vec{x}_r - \vec{x}_\star\bigr)^\top 
        \Biggl[\mathbf{I} - \frac{\mat{A}^\top \mat{A}}{\norm{\mat{A}}_{\rm F}^2}\Biggr]^{s-r}
        \bigl(\vec{x}_r - \vec{x}_\star\bigr)
        \leq (1 - \kappa_{\rm dem}^{-2})^{s-r}\, \lVert \vec{x}_r - \vec{x}_\star \rVert^2,
    \end{equation*}
    for any $r < s$, with probability one.
    The Demmel condition number is $\kappa_{\rm dem} \coloneqq \norm{\smash{\mat{A}^+}}\, \norm{\mat{A}}_{\rm F}$.
\end{lemma}
\begin{proof}
    By the construction of the RK iterates \eqref{eq:rk-update}, observe 
    that $\vec{x}_r$ is in the range of $\mat{A}^\top$, as is the solution vector $\vec{x}_\star = \mat{A}^+ \vec{b} = \mat{A}^\top (\mat{A} \mat{A}^\top)^+ \vec{b}$.
    Hence, $\vec{x}_r - \vec{x}_\star \in \operatorname{range}(\mat{A}^\top)$.
    The result follows by expanding $\vec{x}_r - \vec{x}_\star$ in the basis of $\mat{A}$'s right singular vectors.
\end{proof}

\begin{lemma}[Mean square errors, based on \cite{strohmer2008randomized}]
\label{lem:magnitudes}
    Assume $\vec{x}_0 \in \operatorname{range}(\mat{A}^\top)$.
    Then the RK iteration \eqref{eq:rk} satisfies
    \begin{equation*}
        \expect\, \bigl\lVert \vec{x}_r - \vec{x}_\star \bigr\rVert^2
        \leq \bigl(1 - \kappa_{\rm dem}^{-2}\bigr)^r\, \bigl\lVert \vec{x}_0 - \vec{x}_\star \bigr\rVert^2
        + \norm{\smash{\mat{A}^+}}^2\, \lVert \vec{b} - \mat{A} \vec{x_\star} \rVert^2
    \end{equation*}
    where the expectation averages over the random indices $i_0, i_1, \ldots, i_{r-1}$.
\end{lemma}
\begin{proof}
    For any $t \in \{0, 1, \ldots, r-1\}$, 
    write the one-step update \eqref{eq:rk-update} as
    \begin{equation}
    \label{eq:orthogonal}
    \vec{x}_{t+1} - \vec{x}_\star
    = \underbrace{\biggl[\mathbf{I} - \frac{\vec{a}_{i_t}^{\vphantom{\top}} \vec{a}_{i_t}^\top}{\lVert \vec{a}_{i_t} \rVert^2}\biggr]}_{\text{orthogonal projection}} \bigl(\vec{x}_t - \vec{x}_\star\bigr)
    + \frac{b_{i_t} - \vec{a}_{i_t}^\top \vec{x}_\star}{\norm{\vec{a}_{i_t}}^2} \vec{a}_{i_t}.
    \end{equation}
    The decomposition explicitly identifies an orthogonal projection matrix.
    The matrix is idempotent, it annihilates the vector $\vec{a}_{i_t}$, and it preserves all vectors orthogonal to $\vec{a}_{i_t}$.
    Hence, using the orthogonal decomposition \eqref{eq:orthogonal} it follows
    \begin{equation*}
        \lVert \vec{x}_{t+1} - \vec{x}_\star \rVert^2
        = \bigl(\vec{x}_t - \vec{x}_\star\bigr)^\top \biggl[\mathbf{I} - \frac{\vec{a}_{i_t}^{\vphantom{\top}} \vec{a}_{i_t}^\top}{\lVert \vec{a}_{i_t} \rVert^2}\biggr]
        \bigl(\vec{x}_t - \vec{x}_\star\bigr)
        + \frac{|b_{i_t} - \vec{a}_{i_t}^\top \vec{x}_\star|^2}{\norm{\vec{a}_{i_t}}^2}.
    \end{equation*}
    Use the sampling probabilities \eqref{eq:rk-sample} to calculate the expectation over the random index $i_t$:
    \begin{align*}
        \expect\, \bigl[ \lVert \vec{x}_{t+1} - \vec{x}_\star \rVert^2 \, \big|\, \vec{x}_t \bigr]
        &= \bigl(\vec{x}_t - \vec{x}_\star\bigr)^\top \Biggl[\mathbf{I} - \frac{\mat{A}^\top \mat{A}}{\norm{\mat{A}}_{\rm F}^2}\Biggr]
        \bigl(\vec{x}_t - \vec{x}_\star\bigr)
        + \frac{\lVert \vec{b} - \mat{A} \vec{x}_\star\rVert^2}{\norm{\mat{A}}_{\rm F}^2} \\
        &\leq (1 - \kappa_{\rm dem}^{-2})\, \lVert \vec{x}_t - \vec{x}_\star \rVert^2
        + \frac{\lVert \vec{b} - \mat{A} \vec{x}_\star\rVert^2}{\norm{\mat{A}}_{\rm F}^2},
    \end{align*}
    where the inequality follows from Lemma~\ref{lem:demmel}.
    Next, take the expectation over the random indices $i_0, \ldots, i_t$:
    \begin{equation*}
        \expect\, \lVert \vec{x}_{t+1} - \vec{x}_\star \rVert^2
        \leq (1 - \kappa_{\rm dem}^{-2}) \, \expect\, \lVert \vec{x}_t - \vec{x}_\star \rVert^2
        + \frac{\lVert \vec{b} - \mat{A} \vec{x}_\star\rVert^2}{\norm{\mat{A}}_{\rm F}^2},
        \quad \text{for each } t \in \{0, \ldots, r-1\}.
    \end{equation*}
    Since $\sum_{s=0}^{\infty} (1 - \kappa_{\rm dem}^{-2})^s = \kappa_{\rm dem}^2 = \norm{\smash{\mat{A}^+}}^2 \norm{\mat{A}}_{\rm F}^2$, this equation implies the desired result.
\end{proof}
\begin{proof}[Proof of Theorems~\ref{thm:variance_simple} and \ref{thm:alternative}]
First decompose the mean square error as follows:
\begin{equation*}
    \expect\, \bigl\lVert \overline{\vec{x}}_t - \vec{x}_\star \bigr\rVert^2
    = \frac{1}{(t - t_{\rm b})^2}\sum\nolimits_{r,s=t_{\rm b}}^{t-1} \expect\, \bigl[(\vec{x}_r - \vec{x}_\star)^\top (\vec{x}_s - \vec{x}_\star) \bigr].
\end{equation*}
Next analyze the terms $\expect\, \bigl[(\vec{x}_r - \vec{x}_\star)^\top (\vec{x}_s - \vec{x}_\star) \bigr]$ for $r \leq s$ using Lemmas~\ref{lem:expectations}, \ref{lem:demmel}, and \ref{lem:magnitudes}:
\begin{align*}
    & \expect\, \bigl[(\vec{x}_r - \vec{x}_\star)^\top (\vec{x}_s - \vec{x}_\star) \bigr] \\
    &= \expect\, \bigl[(\vec{x}_r - \vec{x}_\star)^\top \expect\,\bigl[\vec{x}_s - \vec{x}_\star\, \big| \,\vec{x}_r \bigr] \bigr] \\
    &= \expect \Biggl[\bigl(\vec{x}_r - \vec{x}_\star \bigr)^\top \Biggl[\mathbf{I} - \frac{\mat{A}^\top \mat{A}}{\norm{\mat{A}}_{\rm F}^2}\Biggr]^{s-r} (\vec{x}_r - \vec{x}_\star) \Biggr] \\
    &\leq (1 - \kappa_{\rm dem}^{-2})^{s-r}\, \expect\, \lVert \vec{x}_r- \vec{x}_\star \rVert^2 \\
    &\leq \underbrace{\bigl(1 - \kappa_{\rm dem}^{-2}\bigr)^s\, \bigl\lVert \vec{x}_0 - \vec{x}_\star \bigr\rVert^2}_{\text{term A}}
    + \underbrace{\bigl(1 - \kappa_{\rm dem}^{-2}\bigr)^{s-r}\, \norm{\smash{\mat{A}^+}}^2\, \lVert \vec{b} - \mat{A} \vec{x_\star} \rVert^2}_{\text{term B}}.
\end{align*}
By bounding term A uniformly as $\bigl(1 - \kappa_{\rm dem}^{-2}\bigr)^s\, \bigl\lVert \vec{x}_0 - \vec{x}_\star \bigr\rVert^2 \leq \bigl(1 - \kappa_{\rm dem}^{-2}\bigr)^{t_{\rm b}}\, \bigl\lVert \vec{x}_0 - \vec{x}_\star \bigr\rVert^2$ and explicitly averaging over term B, it follows
\begin{align*}
    \expect\, \bigl\lVert \overline{\vec{x}}_t - \vec{x}_\star \bigr\rVert^2
    &= \frac{1}{(t - t_{\rm b})^2}\sum\nolimits_{r,s=t_{\rm b}}^{t-1} \expect\, \bigl[(\vec{x}_r - \vec{x}_\star)^\top (\vec{x}_s - \vec{x}_\star) \bigr] \\
    &\leq \bigl(1 - \kappa_{\rm dem}^{-2}\bigr)^{t_{\rm b}} \bigl\lVert \vec{x}_0 - \vec{x}_\star \bigr\rVert^2
    + \frac{\norm{\smash{\mat{A}^+}}^2\, \lVert \vec{b} - \mat{A} \vec{x_\star} \rVert^2}{(t - t_{\rm b})^2}\sum\nolimits_{r,s=t_{\rm b}}^{t-1} \bigl(1 - \kappa_{\rm dem}^{-2}\bigr)^{\,|s-r|}
\end{align*}
Last, apply the coarse bound
\begin{equation*}
    \sum\nolimits_{r,s=t_{\rm b}}^{t-1} \bigl(1 - \kappa_{\rm dem}^{-2}\bigr)^{|s-r|}
    \leq (t - t_{\rm b}) \biggl[-1 + 2 \sum\nolimits_{s=0}^{\infty} \bigl(1 - \kappa_{\rm dem}^{-2}\bigr)^s\biggr]
    = (t - t_{\rm b})\,(2 \kappa_{\rm dem}^2 - 1),
\end{equation*}
which completes the proof of Theorem~\ref{thm:variance_simple}.
Theorem~\ref{thm:alternative} is proved in a similar way, by explicitly averaging over term A also.
\end{proof}

\subsection{Optimal row-access methods} \label{sec:optimality}
When random noise is injected into $\vec{b}$, it becomes
difficult for any row--access method to converge in its estimates of $\vec{x_{\star}}$.
Building on this phenomenon,
\ref{app:lower} defines a set of challenging least-squares problems where any row-access algorithm leads to a mean square error of $d/t \cdot \norm{\smash{\mat{A}^+}}^2\, \lVert \vec{b} - \mat{A} \vec{x_\star} \rVert^2$ or higher.
In contrast, Theorem~\ref{thm:alternative} guarantees that TARK's mean square error vanishes at a rate of $(2 \kappa_{\rm dem}^2 - 1) / t \cdot \norm{\smash{\mat{A}^+}}^2 \lVert \vec{b} - \mat{A} \vec{x_\star} \rVert^2$ as $t \rightarrow \infty$.
Comparing these two bounds, TARK achieves the optimal $\mathcal{O}(1/t)$ scaling, but the prefactor in TARK's convergence rate is not optimal, since it depends on the square Demmel condition number $\kappa_{\rm dem}^2$.

Looking forward, there are a couple paths toward improving the mean square error of TARK and other row-access methods.
First, preconditioning strategies can be used to reduce the prefactor $\kappa_{\rm dem}^2$ toward the theoretical minimum value of $d$; see \ref{app:preprocessing} for an analysis of optimal preconditioning.
Second, row-access methods can be accelerated through block-wise strategies that process multiple rows simultaneously.
Several block-wise strategies have been proposed \cite{needell2014paved,derezinski2024solving,moorman2020randomized,jain2018parallelizing}, but it is unclear which strategy of this type is most efficient.

\subsection{Numerical demonstration}
\label{sec:numerics}

\begin{figure}
    \centering
    \includegraphics[width=0.98\textwidth]{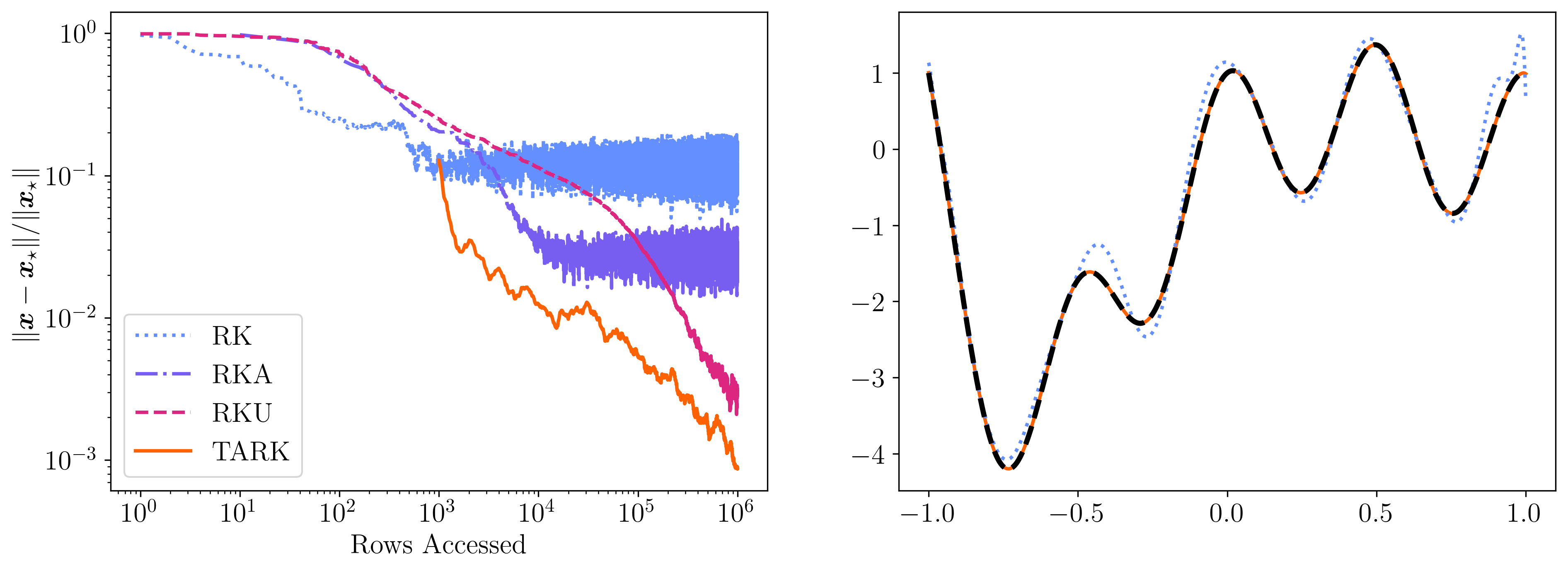}
    \caption{\emph{Left}: Relative errors for four RK methods from Table~\ref{tab:comparison} on a polynomial regression task. 
    \emph{Right}: Computed polynomials for RK and TARK compared to target function.}
    \label{fig:TARK}
\end{figure}

Figure~\ref{fig:TARK} evaluates the performance of four of the RK methods from Table~\ref{tab:comparison} on a polynomial regression task.
The goal is to fit a degree-$(d-1)$ polynomial to $n=10^6$ independent data points $(u_i, b_i)$ where the $u_i$ are equally spaced in $[-1,1]$ and $b_i$ are noisy measurements of a smooth function
\begin{equation*} 
    b_i = f(u_i) + \varepsilon_i,
    \quad \text{where }
    \begin{cases}
        f(u) = \sin(\pi u)\exp(-2u) + \cos(4\pi u), \\
        \varepsilon_i \sim \mathcal{N}(0, 0.04).
    \end{cases}
\end{equation*}
For stability, the polynomial is represented as a linear combination $p = \sum_{j=0}^{d-1} x_j T_j$ of the first $d=25$ Chebyshev polynomials $T_j$. 
The polynomial fitting leads to a $10^6\times 25$ linear least-squares problem with a well-conditioned matrix that satisfies $\norm{\mat{A}}\,\norm{\smash{\mat{A}^+}} < 6$.
This problem is highly overdetermined, but it is small enough to compute an exact reference solution.
See \url{https://github.com/eepperly/Randomized-Kaczmarz-with-Tail-Averaging} for code for all experiments in this paper.

The left panel of Figure~\ref{fig:TARK} compares four of the row-access methods from Table~\ref{tab:comparison}, while the fifth method of extended RK is omitted because it requires column access.
For all four methods, the total number of rows accessed is $t=10^6$, which is equivalent to a single pass over the input data.
The TARK burn-in time is set to $t_{\rm b} = 10^3$, the RKU underrelaxation parameter is $1/\sqrt{t}$, and the number of threads for RKA is $10$.
The results verify that TARK converges past the finite horizon of RK and RKA.
RKU similarly breaks through the finite horizon, but its convergence rate is slower than TARK's rate.
The right panel of Figure~\ref{fig:TARK} demonstrates that the polynomial computed by TARK 
accurately reproduces the target function $f$, whereas the polynomial found by RK exhibits noticeable discrepancies.

\subsection{Extension: semi-infinite problems}\label{sec:semi}
TARK can also be applied to semi-infinite (infinitely tall, finitely wide) least-squares problems \cite{shustin2022semi}
\begin{equation*}
    \min_{\vec{x} \in \real^d} \int_\Omega \big(b(u) - \vec{a}(u)^\top \vec{x}\big)^2 \, \mathrm{d} \nu(u),
\end{equation*}
where $(\Omega,\nu)$ is an arbitrary measure space and $\vec{a} : \Omega \to \real^d$ and $b : \Omega \to \real$ are $\set{L}_2$ functions.
The procedure is completely the same:
\begin{enumerate}
    \item Sample $u_t \sim  \norm{\vec{a}(u)}^2/\norm{\vec{a}}_{\rm F}^2 \, \mathrm{d} \nu(u)$ where $\norm{\vec{a}}_{\rm F}^2 = \int_\Omega \norm{\vec{a}(u)}^2\, \mathrm{d}\nu(u)$.
    \item Update $\vec{x}_{t+1} \coloneqq \vec{x}_t + (b(u_t) - \vec{a}(u_t)^\top \vec{x}_t) \,\vec{a}(u_t) / \norm{\vec{a}(u_t)}^2$.
\end{enumerate}
The natural analog of Theorem~\ref{thm:variance_simple} holds with the same proof.
Row-access methods are especially natural in the semi-infinite setting, as infinite columns cannot be directly stored in finite memory.

\section{Ridge regression}\label{sec:rr}
\newcommand{\rkrr}{RK-RR}
The least-squares problem \eqref{eq:ls} can be regularized by adding a ridge penalty $\lambda\, \norm{\vec{x}}^2$:
\begin{equation} \label{eq:rr}
    \min_{\vec{x} \in \real^d}\, \norm{\vec{b} - \mat{A}\vec{x}}^2 + \lambda\, \norm{\vec{x}}^2 \quad \text{for } \mat{A} \in \real^{n\times d} \text{ and } \vec{b} \in \real^n \text{ with } \lambda > 0,\: n > d.
\end{equation}
Adding this term accelerates convergence when the matrix $\mat{A}$ is ill-conditioned, and it may reduce the impact of noise in the data $(\mat{A},\vec{b})$.
The unique solution to the ridge-regularized problem \eqref{eq:rr} is $\smash{\bigl(\mat{A}^\top \mat{A} + \lambda\, \Id \bigr)}^{-1} \mat{A}^\top \vec{b}$, which can be quite different from the ordinary least-squares solution.
Whether or not adding regularization is appropriate depends on the application. 

To compute the ridge-regularized solution, several variants of RK have been suggested:
\begin{itemize}
\item RK can be modified to solve a consistent linear system involving the solution vector $\vec{x} \in \real^d$ and a dual variable $\vec{y} \in \real^n$ \cite{ ivanov2013kaczmarz, hefny2017rows}.
However, this approach requires storing and manipulating the length-$n$ vector
$\vec{y}$, and it also requires multiple passes over the input data.
Both requirements are computationally taxing for the largest systems.
\item
RK can be applied to the augmented least-squares problem \cite{bautu2005tikhonov}
\begin{equation} 
    \label{eq:rr_tall_system}
    \min_{\vec{x} \in \real^d}\, \norm{\,\begin{bmatrix}
        \vec{b} \\ \vec{0}
    \end{bmatrix} - \begin{bmatrix}
        \mat{A} \\ \sqrt{\lambda}\, \Id_d
    \end{bmatrix} \vec{x}\,}^2,
    \end{equation}
and TARK is also an option for solving this system.
However, this approach with either RK or TARK leads to limited accuracy because it treats the regularization term $\lambda\, \norm{\vec{x}}^2$ stochastically; see Subsection~\ref{sec:numerics2} for further discussion.
\end{itemize}

A different, natural approach to the ridge-regularized problem \eqref{eq:rr} was suggested two decades ago for the task of image reconstruction \cite{popa2005penalized, bautu2005tikhonov}. The main idea is to combine stochastic RK iterations for the least-squares term $\norm{\vec{b} - \mat{A} \vec{x}}^2$ with deterministic gradient descent steps for the regularization term $\lambda\, \norm{\vec{x}}^2$. This same idea forms the basis for \textit{weight decay}, which is commonly used to incorporate ridge regularization when training deep neural networks \cite{goodfellow2016deep}.

In the context of RK, one version of the weight decay approach can be written:
\begin{equation}
\label{eq:alternating}
    \vec{x}_{t+1/2} \coloneqq \vec{x}_t + \frac{b_{i_t} - \vec{a}_{i_t}^\top \vec{x}_t}{\norm{\vec{a}_{i_t}}^2} \vec{a}_{i_t},
    \qquad
    \vec{x}_{t+1} \coloneqq \mu\, \vec{x}_{t+1/2}.
\end{equation}
The parameter $\mu \in (0,1)$ controls the amount of regularization, resulting in the ridge parameter $\lambda = (1-\mu)/\mu \cdot \norm{\mat{A}}_{\rm F}^2$.
We call the scheme \eqref{eq:alternating} randomized Kaczmarz for ridge regression (RK-RR).
Similar to RK, RK-RR converges up to a finite horizon:
\begin{theorem}[Randomized Kaczmarz for ridge regression: convergence to a horizon]
\label{thm:rkrr}
    Assume $\vec{x}_0 \in \operatorname{range}(\mat{A}^\top)$.
    Then RK-RR \eqref{eq:alternating} converges to the ridge-regularized solution
    \begin{equation}
    \label{eq:solution}
    \vec{x}_\mu = \operatornamewithlimits{argmin}_{\vec{x} \in \real^d}\, \Bigl[\norm{\vec{b} - \mat{A}\vec{x}}^2 + \lambda \norm{\vec{x}}^2 \Bigr] \quad \text{for } \lambda = \frac{1-\mu}{\mu} \norm{\mat{A}}_{\rm F}^2
\end{equation}
at an exponential rate, up to a finite horizon related to the residual:
\begin{equation*}
    \expect\, \bigl\lVert \vec{x}_t - \vec{x}_\mu \bigr\rVert^2
    \leq 2\, [\mu^2 (1 - \kappa_{\rm dem}^{-2})]^t \, \bigl\lVert \vec{x}_0 - \vec{x}_\mu \bigr\rVert^2
    + \frac{2\mu}{(1 + \mu) \lambda} \, \lVert \vec{b} - \mat{A} \vec{x}_\mu \rVert^2.
\end{equation*}
\end{theorem}
Compared to the error bounds for RK, the regularization plays a key role in speeding up the convergence and controlling the size of the horizon.
The proof of Theorem~\ref{thm:rkrr} can be found in \ref{app:TARK-RR_proof}.

\begin{algorithm}[t]
\caption{Tail-averaged randomized Kaczmarz for ridge regression (TA\rkrr{})} \label{alg:TA\rkrr{}}
\begin{algorithmic}[1]
\Require Matrix $\mat{A}$, vector $\vec{b}$, initial estimate $\vec{x}_0 \in \real^d$, regularization $\mu$, burn-in time $t_{\rm b}$, and final time $t$
\For{$s$ in $0, \ldots, t-2$}
    \State Sample $i \sim \norm{\vec{a}_i}^2 /\norm{\mat{A}}_{\rm F}^2$ 
    \State $\vec{x}_{s+1/2}= \vec{x}_s + (b_i - \vec{a}_i^\top \vec{x}_s)\,\vec{a}_i / \norm{\vec{a}_i}^2$
    \State $\vec{x}_{s+1}= \mu\, \vec{x}_{s+1/2}$
    \EndFor 
\State $\overline{\vec{x}}_t = \bigl(\sum_{s = t_{\rm b}}^{t-1} \vec{x}_s\bigr) / (t-t_{\rm b})$ \\
\Return $\overline{\vec{x}}_t$ 
\end{algorithmic}
\end{algorithm}

\subsection{Tail-averaged randomized Kaczmarz for ridge regression}
Similar to RK, the finite horizon of RK-RR can be overcome by using tail averaging. The resulting method is
\textit{tail-averaged randomized Kaczmarz for ridge regression} (TARK-RR);
see Algorithm~\ref{alg:TA\rkrr{}}.
The following theorem quantifies the convergence rate of TA\rkrr{}:
\begin{theorem}[Mean square error for TA\rkrr{}]
\label{thm:variance_mu}
Assume $\vec{x}_0 \in \operatorname{range}(\mat{A}^\top)$ and recall that the ridge parameter is  $\lambda = (1-\mu)/\mu \cdot \norm{\mat{A}}_{\rm F}^2$.
Then Algorithm~\ref{alg:TA\rkrr{}} converges to the ridge-regularized solution $\vec{x}_\mu$ \eqref{eq:solution} at a rate that balances exponential and polynomial convergence:
\begin{equation*}
    \expect\, \norm{\overline{\vec{x}}_t - \vec{x}_\mu}^2 \leq 
    \underbrace{2\, [\mu^2
    (1 - \kappa_{\rm dem}^{-2})]^{t_{\rm b}} \, \norm{\vec{x}_0 - \vec{x}_\mu}^2}_{\text{exponential convergence}} 
    + \underbrace{ \frac{2 \mu}{(t - t_{\rm b}) (1-\mu) \lambda}\, \norm{\vec{b} - \mat{A}\vec{x}_\mu}^2}_{\text{polynomial convergence}}.
\end{equation*}
\end{theorem}
The proof of Theorem~\ref{thm:variance_mu} can be found in \ref{app:TARK-RR_proof}.  See Table~\ref{tab:comparison-rr} for a comparison of TARK-RR with other RK-based approaches. 

\begin{table}[t]
    \centering
    \begin{tabular}{lcccccl}\toprule
     Method & Final rate & Handling of $\lambda\, \norm{\vec{x}}^2$ & Length-$d$ vectors? \\  \midrule
    Dual methods  \cite{ivanov2013kaczmarz, hefny2017rows} & \textcolor{goodcolor}{Exponential}  & \textcolor{goodcolor}{Deterministic} & \textcolor{badcolor}{No \xmark} \\ \midrule
    RK on \eqref{eq:rr_tall_system}& \textcolor{badcolor}{Finite horizon} &  \textcolor{badcolor}{Stochastic} & \textcolor{goodcolor}{Yes \cmark}  \\
    TARK on \eqref{eq:rr_tall_system} & \textcolor{goodcolor}{Polynomial}  & \textcolor{badcolor}{Stochastic}  & \textcolor{goodcolor}{Yes \cmark} \\  \midrule 
    \rkrr{} & \textcolor{badcolor}{Finite horizon} & \textcolor{goodcolor}{Deterministic}  & \textcolor{goodcolor}{Yes \cmark} 
    \\ 
    TA\rkrr{} & \textcolor{goodcolor}{Polynomial} & \textcolor{goodcolor}{Deterministic}  & \textcolor{goodcolor}{Yes \cmark} 
    \\ \bottomrule 
    \end{tabular}
    \caption{RK variants for ridge regression problems. The table lists the final convergence rate, how the regularization is handled, and whether the method only manipulates length-$d$ vectors.}
    \label{tab:comparison-rr}
\end{table}

\subsection{Numerical demonstration} \label{sec:numerics2}

This section repeats the polynomial regression experiment from Section~\ref{sec:numerics} but uses an unstable representation of the regression polynomial $p(u) = \sum_{j=0}^{d-1} x_j u^j$ as a linear combination of monomials.
This change of representation leads to an ill-conditioned matrix with condition number $\norm{\mat{A}}\, \norm{\smash{\mat{A}^+}} \approx 6\times 10^8$.

\begin{figure}
    \centering
    \includegraphics[width=\textwidth]{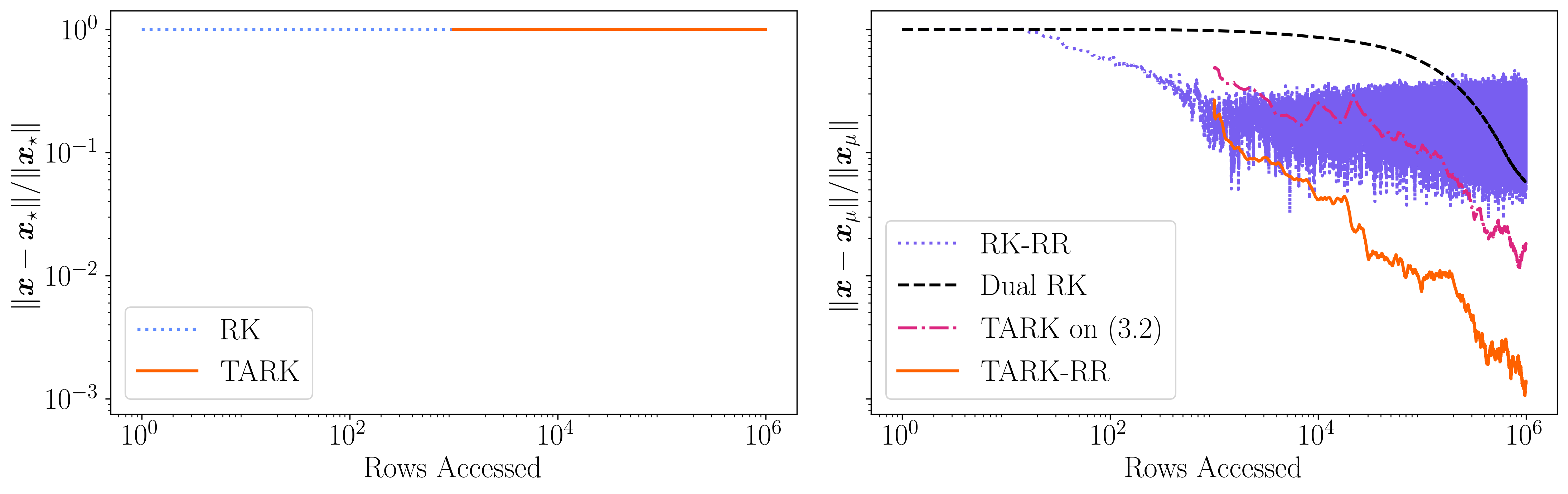}
    \caption{Relative errors for RK methods applied to un-regularized (\emph{left}) and regularized (\emph{right}) polynomial regression problems.}
    \label{fig:TA\rkrr{}}
\end{figure}

The left panel of Figure~\ref{fig:TA\rkrr{}} demonstrates that RK and TARK converge extremely slowly for the ordinary least-squares system \eqref{eq:ls}, motivating the need for regularization.
The right panel of Figure~\ref{fig:TA\rkrr{}} shows the results of adding ridge-regularization with $\mu=0.999$.
This approach changes the solution and enables TA\rkrr{} to make faster progress than the unregularized methods. 
Also pictured are the dual RK method of \cite{hefny2017rows} and TARK applied to the augmented system \eqref{eq:rr_tall_system}.
These algorithms make significantly less progress than TARK-RR, providing evidence that approaches based on dual variables or the augmented system \eqref{eq:rr_tall_system} are not competitive for highly overdetermined linear least-squares problems.

In summary, the experiments suggest that the alternating minimization \eqref{eq:alternating} is the most effective way of incorporating ridge regularization into Kaczmarz-type algorithms for linear least-squares problems.
This observation may have implications for nonlinear optimization, including the recently proposed SPRING algorithm for variational Monte Carlo simulation \cite{goldshlager2024kaczmarz}. 

\section{Conclusion}

Randomized Kaczmarz has long served as a simple, explicitly analyzable algorithm that responds to the key scaling challenges of overdetermined linear least-squares problems.
In addition, the detailed analysis of randomized Kaczmarz has highlighted broader opportunities to understand and improve stochastic gradient descent methods \cite{needell2014stochastic}.
Building on this past research, the current work highlights the opportunity to incorporate tail averaging within randomized Kaczmarz to improve the convergence rate.
These results are encouraging regarding the use of tail averaging, even beyond the linear least-squares problem.
As further opportunities, this paper points toward preconditioning and block-wise arithmetic as opportunities to speed up the performance of large-scale linear least-squares solvers,
and it suggests alternating minimization as an effective technique for incorporating ridge regularization into Kaczmarz-type algorithms.

\section*{Acknowledgments}
The authors thank Haoxuan Chen, Zhiyan Ding, Micha{\l} Derezi{\'n}ski, Jamie Haddock, Jiang Hu, Lin Lin, Anna Ma, Christopher Musco, Deanna Needell, Kevin Shu, Joel Tropp, Roman Vershynin, and Jonathan Weare for helpful conversations.
ENE acknowledges support from the Department of Energy under Award Number DE-SC0021110 and, under the aegis of Joel Tropp, from NSF FRG 1952777 and the Carver Mead New Horizons Fund.
GG acknowledges support from the Department of Energy under Award Number DE-SC0023112.

\appendix

\section{TARK with increasing burn-in time} \label{app:increase}
To control both bias and variance, we recommend implementing TARK with a burn-in time that comprises a quarter to half of the final time, $t_{\rm b} \in [t/4,t/2]$.
In practice, we may not know the final time $t$ in advance, opting to run the algorithm for as many iterations as needed for the solution to meet some error tolerance.
To implement TARK in a storage-efficient manner in this setting, one can use the following TARK implementation with an increasing burn-in time $t_{\rm b} = 2^{\lfloor \log_2(t) \rfloor - 1}$.
The approach is based on storing just two extra vectors $\tilde{\vec{x}}_{\rm old}, \tilde{\vec{x}}_{\rm new} \in \real^d$, where
$\tilde{\vec{x}}_{\rm old}$ sums the TARK iterates from time $2^{\lfloor \log_2(t) \rfloor - 1}$ to $2^{\lfloor \log_2(t) \rfloor}$ and $\tilde{\vec{x}}_{\rm new}$ sums the TARK iterates from time $2^{\lfloor \log_2(t) \rfloor}$ to $t$.
When the iteration time $t$ hits a power of $2$, the two vectors are updated via $\tilde{\vec{x}}_{\rm old} \leftarrow \tilde{\vec{x}}_{\rm new}$ and $\tilde{\vec{x}}_{\rm new} \leftarrow \vec{0}$.
At any time $t$, the TARK result can be quickly calculated using $\tilde{\vec{x}}_{\rm old}$ and $\tilde{\vec{x}}_{\rm new}$ as follows:
\begin{equation*}
    \tilde{\vec{\vec{x}}}_t = \frac{1}{t-t_{\rm b}} \sum\nolimits_{s = t_{\rm b}}^{t-1} \vec{x}_s
    = \frac{\tilde{\vec{x}}_{\rm old} + \tilde{\vec{x}}_{\rm new}}{t - t_{\rm b}}.
\end{equation*}
See Algorithm~\ref{alg:TARK_increasing} for pseudocode.

\begin{algorithm}[t]
\caption{TARK with increasing burn-in time} \label{alg:TARK_increasing}
\begin{algorithmic}[1]
\Require Matrix $\mat{A} \in \real^{n \times d}$, vector $\vec{b} \in \real^n$, initial estimate $\vec{x}_0 \in \real^d$, and final time $t$
\State $\tilde{\vec{x}}_{\rm old} = \vec{0}$, $\tilde{\vec{x}}_{\rm new} = \vec{0}$
\For{$s$ in $0, \ldots, t-2$}
    \State Sample $i \sim \norm{\vec{a}_i}^2 /\norm{\mat{A}}_{\rm F}^2$ 
    \State $\vec{x}_{s+1}= \vec{x}_s + (b_i - \vec{a}_i^\top \vec{x}_s) \,\vec{a}_i / \norm{\vec{a}_i}^2$
    \State $\tilde{\vec{x}}_{\rm new} = \tilde{\vec{x}}_{\rm new} + \vec{x}_{s+1}$
    \If{$s + 1$ is a power of $2$}
        \State $\tilde{\vec{x}}_{\rm old} = \tilde{\vec{x}}_{\rm new}$, $\tilde{\vec{x}}_{\rm new} = \vec{0}$
    \EndIf
\EndFor
\State $t_{\rm b} = 2^{\lfloor \log_2(t) \rfloor - 1}$
\State $\overline{\vec{x}}_t = (\tilde{\vec{x}}_{\rm old} + \tilde{\vec{x}}_{\rm new}) / (t-t_{\rm b})$ \\
\Return $\overline{\vec{x}}_t$ 
\end{algorithmic}
\end{algorithm}

\section{Lower bounds} \label{app:lower}
The following proposition constrains the best performance that a row-access method can attain.
\begin{proposition}[Lower bound on mean square error] \label{prop:need_polynomial}
    Fix $\varepsilon > 0$ and $d \geq 1$.
    Any algorithm that can solve all least-squares problem $\min_{\vec{x}}\,\lVert \vec{b} - \mat{A} \vec{x} \rVert^2$ with mean square error
    \begin{equation}
    \label{eq:guarantee}
        \expect\, \lVert \hat{\vec{x}} - \vec{x}_\star \rVert^2 \leq \varepsilon \, \norm{\smash{\mat{A}^+}}^2\, \lVert \vec{b} - \mat{A} \vec{x_\star} \rVert^2,
    \end{equation}
    must allow access to $t \geq d / \varepsilon$ entries of $\vec{b}$.
\end{proposition}
Previous results \cite{chen2019active} have demonstrated the same $t = \Omega(d/\varepsilon)$ scaling; see also the discussion in \cite[sec.~1.1.4]{jain2018parallelizing}.

\begin{proof}
Consider applying a least-squares solver to
a a random least-squares problem, where
\begin{equation}
\label{eq:random_class}
    \mat{A} = \begin{bmatrix}
        \vec{1}_m & \cdots & \vec{0}_m \\
        \vdots & \ddots &\vdots  \\
        \vec{0}_m & \cdots & \vec{1}_m
    \end{bmatrix} \in \real^{m d \times d},
    \quad
    \vec{b} = \begin{bmatrix}
        \vec{b}_1 \\
        \vdots \\
        \vec{b}_d
    \end{bmatrix} \in \real^{md},
    \quad \text{and }
    \vec{b}_i \sim \text{iid }  \mathcal{N}\bigl(\vec{0}, \mat{\Sigma} \bigr),
    \text{  for } \mat{\Sigma} = \Id_m + v \vec{1}_m \vec{1}_m^\top.
\end{equation}
Here, $\vec{0}_m, \vec{1}_m \in \real^m$ are the vectors of all zeroes and all ones, $m$ controls the aspect ratio in the problem, and $v > 0$ is a variance parameter.
Each problem decomposes into the sum of $d$ simpler problems:
\begin{equation*}
    \min_{\vec{x} \in \real^d} \, \norm{\vec{b} - \mat{A} \vec{x}}^2
    = \min_{\vec{x} \in \real^d}
    \sum\nolimits_{i = 1}^d \lVert \,\vec{b}_i - \vec{1}_m x_i\, \rVert^2.
\end{equation*}
Hence, each entry of the least-squares solution $\vec{x}_\star \in \real^d$ is an arithmetic mean of $m$ entries of the output vector: 
\begin{equation*}
    \vec{x}_\star = \begin{bmatrix} x_{\star, 1} & \cdots & x_{\star, d} \end{bmatrix}^\top = \begin{bmatrix} \frac{1}{m} \vec{1}_m^\top \vec{b}_1 & \cdots & \frac{1}{m} \vec{1}_m^\top \vec{b}_d
    \end{bmatrix}^\top.
\end{equation*}
In this random problem class, 
a direct calculation using the Gaussian covariance matrix $\mat{\Sigma} = \Id_m + v \vec{1}_m \vec{1}_m^\top$ shows
\begin{equation*}
    \expect\, \norm{\vec{b} - \mat{A} \vec{x_\star}}^2
    = \sum\nolimits_{i=1}^d \expect\, \Bigl\lVert  \vec{b}_i - \tfrac{1}{m} \vec{1}_m \vec{1}_m^\top \vec{b}_i \Bigr\rVert^2
    = d (m-1).
\end{equation*}
This is the mean square error of the least-squares solution.
Further, the matrix $\mat{A}$ has singular values $\sigma_i = \sqrt{m}$ for $i = 1, \ldots, m$, so $\lVert \mat{A}^+ \rVert^2 = \frac{1}{m}$.

Now suppose a least-squares solver accesses certain entries of $\vec{b}_1$ that are indexed by $\set{S}_1 \subseteq \{1, \ldots, m\}$, certain entries of $\vec{b}_2$ that are indexed by $\set{S}_2 \subseteq \{1, \ldots, m\}$, and so on.
Given a subset of $k = |\textsf{S}_i|$ revealed entries of $\vec{b}_i$, evaluate the Gaussian conditional mean and variance formulas for the unrevealed entries $\vec{b}_{i, \textsf{S}_i^c}$ as follows:
\begin{equation*}
    \mathbb{E}\,\bigl[\vec{b}_{i, \textsf{S}_i^c} \,\big|\, \vec{b}_{i, \textsf{S}_i} \bigr]
    = \mat{\Sigma}_{\textsf{S}_i^c, \textsf{S}_i}\, \mat{\Sigma}_{\textsf{S}_i, \textsf{S}_i}^{-1}\,
    \vec{b}_{i,\textsf{S}_i}
    = \frac{v}{1 + vk} \vec{1}_{m-k} \vec{1}_k^\top \vec{b}_{i,\textsf{S}_i}.
\end{equation*}
and also
\begin{equation*}
    \operatorname{cov}\,\bigl[\vec{b}_{i, \textsf{S}_i^c} \,|\, \vec{b}_{i, \textsf{S}_i}]
    = \mat{\Sigma}_{\textsf{S}_i, \textsf{S}_i} - \mat{\Sigma}_{\textsf{S}_i^c, \textsf{S}_i}\, 
    \mat{\Sigma}_{\textsf{S}_i, \textsf{S}_i}^{-1}\,
    \mat{\Sigma}_{\textsf{S}_i, \textsf{S}_i^c}
    = \Id_{m-k} + \frac{v}{1 + vk} \vec{1}_{m-k} \vec{1}_{m-k}^\top.
\end{equation*}
Therefore, the conditional expectation formula for $x_{\star, i}$ is
\begin{equation*}
    \mathbb{E}\,\bigl[ x_{\star, i} \,\big|\, \vec{b}_{i, \textsf{S}_i} \bigr]
    = \mathbb{E}\,\biggl[\frac{1}{m} \vec{1}_m^\top  \vec{b}_i \,\bigg|\, \vec{b}_{i, \textsf{S}_i} \biggr]
    = \frac{1}{m} \biggl[1 + \frac{v(m-k)}{1 + vk} \biggr] \vec{1}_k^\top \vec{b}_{i, \textsf{S}_i}. 
\end{equation*}
The conditional variance formula for $x_{\star, i}$ is
\begin{equation*}
    \operatorname{Var}\,\bigl[x_{\star, i} \,\big|\, \vec{b}_{i, \textsf{S}_i} \bigr]
    = \operatorname{Var}\,\biggl[\frac{1}{m} \vec{1}_m^\top \vec{b}_i \,\bigg|\, \vec{b}_{i, \textsf{S}_i} \biggr]
    = \frac{\vec{1}_{m-k}^\top \operatorname{cov}\,\bigl[\vec{b}_{i, \textsf{S}_i^c} \,\big|\, \vec{b}_{i, \textsf{S}_i} \bigr]\, \vec{1}_{m-k}}{m^2}
    = \frac{m - k}{m^2}\biggl[1 + \frac{v (m-k)}{1 + vk} \biggr]
\end{equation*}
The vector with entries $\mu_i = \mathbb{E}\,\bigl[ x_{\star, i} \,\big|\, \vec{b}_{i, \textsf{S}_i} \bigr]$ optimizes the mean square error of approximating $\vec{x}_\star$.
\begin{equation*}
    \mu_i = \operatornamewithlimits{argmin}_{\hat{x}_i} \expect\,\bigl[ |\,\hat{x}_i - x_{\star, i} |^2 \,|\, \vec{b}_{i, \set{S}_i} \bigr].
\end{equation*}
Therefore, taking the expectation over the unrevealed entries of $\vec{b}$, it holds for any estimator $\hat{\vec{x}}$: 
\begin{align*}
    \expect\, \bigl[ \lVert \,\hat{\vec{x}} - \vec{x}_\star \rVert^2 \,|\, \vec{b}_{1, \set{S}_1}, \ldots, \vec{b}_{d, \set{S}_d} \bigr]
    &\geq
    \expect\, \bigl[ \lVert\, \vec{\mu} - \vec{x}_\star \rVert^2 \,|\, \vec{b}_{1, \set{S}_1}, \ldots, \vec{b}_{d, \set{S}_d} \bigr] \\
    & = \frac{1}{m^2} \sum\nolimits_{i=1}^d (m - |\textsf{S}_i|)\,\biggl[1 + \frac{v(m-|\textsf{S}_i|)}{1 + v |\textsf{S}_i|} \biggr].
\end{align*}
Now let $t$ be the maximum number of entries accessed,
and observe that $x \mapsto (m - x)\bigl[1 + v(m - x) / (1 + vx)\bigr]$ is convex and decreasing, so
the conditional mean square error is bounded from below by setting $|\textsf{S}_i| = t/d$ for each $i \in \{1, \ldots, d\}$:
\begin{equation*}
    \expect\, \bigl[ \lVert \hat{\vec{x}} - \vec{x}_\star \rVert^2 \,|\, \vec{b}_{1, \set{S}_1}, \ldots, \vec{b}_{d, \set{S}_d} \bigr] 
    \geq \frac{d}{m^2} (m - t/d)\,\biggl[1 + \frac{v(m - t/d)}{1 + v t / d} \biggr]
\end{equation*}
By averaging over the revealed entries, the mean square error satisfies the same error bound:
\begin{equation*}
    \expect\, \lVert \hat{\vec{x}} - \vec{x}_\star \rVert^2 \geq \frac{d}{m^2} (m - t/d)\,\biggl[1 + \frac{v(m - t/d)}{1 + v t / d} \biggr].
\end{equation*}
This is the minimal least-squares solver that an algorithm can possibly achieve after revealing $t$ entries of $\vec{b}$.

Now suppose that a least-squares solver satisfies the error bound \eqref{eq:guarantee} for each linear least-squares problem.
Then apply the least-squares solver to the random problem class \eqref{eq:random_class} and average the resulting error bounds \eqref{eq:guarantee} to yield
\begin{equation}
\label{eq:the_relation}
    \frac{d}{m^2} (m - t/d)\,\biggl[1 + \frac{v(m - t/d)}{1 + v t / d} \biggr]
    \leq \expect\, \lVert \hat{\vec{x}} - \vec{x}_\star \rVert^2
    \leq \varepsilon \,\norm{\smash{\mat{A}^+}}^2\,
    \expect\, \lVert \vec{b} - \mat{A} \vec{x_\star} \rVert^2
    = \varepsilon \frac{d(m-1)}{m}.
\end{equation}
The relation \eqref{eq:the_relation} must hold for any parameters $m$ and $v$.
Sending $v \rightarrow \infty$ implies that
\begin{equation*}
    t
    \geq \frac{d^2}{\varepsilon d + (1 - \varepsilon) \frac{d}{m}}.
\end{equation*}
Next, sending $m \rightarrow \infty$ implies that the maximal number of entries which need to be accessed by the algorithm is $t \geq d / \varepsilon$.
\end{proof}

Proposition~\ref{prop:need_polynomial} also leads to a bound on the mean square residual error $\expect\, \lVert  \vec{b} - \mat{A} \hat{\vec{x}} \rVert^2$.

\begin{corollary}[Lower bound on mean square residual error] \label{cor:need_polynomial}
    Fix $\varepsilon > 0$ and $d \geq 1$.
    Any algorithm that can solve all least-squares problem $\min_{\vec{x}}\,\lVert \vec{b} - \mat{A} \vec{x} \rVert^2$ with mean square residual error
    \begin{equation}
    \label{eq:equivalent}
        \expect\, \lVert \vec{b}  - \mat{A} \hat{\vec{x}} \rVert^2 \leq (1 + \varepsilon) \,\lVert \vec{b} - \mat{A} \vec{x_\star} \rVert^2
    \end{equation}
    must allow access to $t \geq d / \varepsilon$ entries of $\vec{b}$.
\end{corollary}
\begin{proof}
    By an orthogonal decomposition, $\lVert \vec{b}  - \mat{A} \hat{\vec{x}} \rVert^2
    = \lVert  \vec{b}  - \mat{A} \vec{x}_\star \rVert^2  + \lVert \mat{A} \hat{\vec{x}}  -  \mat{A} \vec{x}_\star \rVert^2$.
    Thus, \eqref{eq:equivalent} can be rewritten as
    \begin{equation*}
        \expect\, \lVert \mat{A}  \hat{\vec{x}} - \mat{A}\vec{x}_\star \rVert^2 \leq \varepsilon \, \lVert \vec{b} - \mat{A} \vec{x_\star} \rVert^2.
    \end{equation*}
    Any algorithm that guarantees \eqref{eq:equivalent} also guarantees
    \begin{equation*}
        \expect\, \lVert  \mat{A}^+ \mat{A} \hat{\vec{x}} - \vec{x}_\star  \rVert^2 
        = \expect\, \lVert  \mat{A}^+ (\mat{A} \hat{\vec{x}} - \mat{A} \vec{x}_\star) \rVert^2 
        \leq \norm{\smash{\mat{A}^+}}^2 \expect\,\lVert \mat{A} \hat{\vec{x}} - \mat{A} \vec{x}_\star \rVert^2
        \leq \varepsilon \,\norm{\smash{\mat{A}^+}}^2 \lVert \vec{b} - \mat{A} \vec{x_\star}  \rVert^2.
    \end{equation*}
    By Proposition~\ref{prop:need_polynomial}, the algorithm must allow access to $t \geq d / \varepsilon$ entries of $\vec{b}$.
\end{proof}

\section{Achieving the lower bounds: Preconditioning and initialization} \label{app:preprocessing}
We recognize two areas of improvement for TARK.
First, the TARK mean square error bound in Theorem~\ref{thm:variance_simple} depends on the square Demmel condition number $\kappa_{\rm dem}^2$,
whereas the lower bound in Proposition~\ref{prop:need_polynomial} depends on the dimension $d$, which is always smaller.
Second, the TARK error bound suggests a burn-in period is needed to wash out the influence of the initialization $\vec{x}_0$.
The first problem can by addressed by applying TARK to a preconditioned version of the least-squares problem
\begin{equation*}
    \vec{y}_\star = \argmin_{\vec{y} \in \real^d} \,\norm{\vec{b} - (\mat{A}\mat{R}^{-1})\vec{y} }^2;\quad \vec{x}_\star = \mat{R}^{-1}\vec{y}_\star.
\end{equation*}
The second problem can be addressed using a careful choice of $\vec{x}_0 \approx \vec{x}_\star$.

Both preconditioning and finding a high-quality initialization can be computationally expensive, perhaps prohibitively expensive when $\mat{A}$ is large.
Nevertheless, the following result demonstrates that, given the computational resources to compute these objects, even a simple row-access method like TARK can achieve near-optimal results:

\begin{theorem}[Preconditioned TARK with volume sampling]\label{prop:active_learning}
Given a matrix $\mat{A} \in \real^{n \times d}$ of rank $r$ and a vector $\mat{b} \in \real^n$, consider the following algorithm:
\begin{enumerate}
    \item Calculate a thin QR decomposition $\mat{A} = \mat{Q} \mat{R}$ for $\mat{Q} \in \real^{n\times r}$.
    \item Sample a subset of $r$ rows $\set{S} \subseteq \{1, \ldots, n\}$ from the square-volume distribution \cite{derezinski2018reverse}
    \begin{equation*}
        \mathbb{P}(\set{S}) = \frac{\operatorname{det}(\mat{Q}(\set{S},:))^2}{\sum\nolimits_{|\set{S}'| = r} \operatorname{det}(\mat{Q}(\set{S'},:))^2}.
    \end{equation*}
    \item Apply TARK with the initial estimator $\vec{y}_0 = \mat{Q}_{\set{S}}^{-1} \vec{b}_{\set{S}}$ to solve $\min_{\vec{y}}\, \lVert \vec{b} - \mat{Q} \vec{y} \rVert^2$.
    \item Solve the triangular system $\hat{\vec{x}} = \mat{R}^+\,\overline{\vec{y}}_t$, where $\overline{\vec{y}}_t$ is the output vector from TARK.
\end{enumerate}
Then, the TARK-based solution $\hat{\vec{x}} \in \real^d$ satisfies
\begin{equation*}
    \expect\, \bigl\lVert \vec{b}  - \mat{A} \hat{\vec{x}} \bigr\rVert^2
    \leq \biggl[1 + \biggl(1 - \frac{1}{r}\biggr)^{t_{\rm b}} r + \frac{2r - 1}{t - t_{\rm b}} \biggr] \,\bigl\lVert \vec{b} - \mat{A} \vec{x}_\star \bigr\rVert^2,
\end{equation*}
where $t_{\rm b}$ and $t$ the burn-in time and final time used in TARK. In particular, setting $t_{\rm b} = t/2$, this algorithm achieves the guarantee $\expect \lVert \vec{b} - \mat{A} \hat{\vec{x}}\rVert^2 \leq (1 + \varepsilon) \cdot \lVert \vec{b} - \mat{A} \vec{x_\star} \rVert^2$
after evaluating just
\begin{equation*}
    t = r + r \log\biggl(\frac{2r}{\varepsilon}\biggr) + \frac{4r - 2}{\varepsilon} \text{ entries of $\vec{b}$.}
\end{equation*}
\end{theorem}

\begin{proof}
Because $\mat{Q}$ has orthonormal columns, every vector $\vec{y} \in \real^r$ satisfies
\begin{equation*}
    \bigl\lVert \vec{b} - \mat{Q} \vec{y} \bigr\rVert^2
    = \bigl\lVert \vec{b}  - \mat{Q} \vec{y}_\star \bigr\rVert^2 + \bigl\lVert \mat{Q} \vec{y} - \mat{Q} \vec{y}_\star \bigr\rVert^2 
    = \bigl\lVert \vec{b} - \mat{Q} \vec{y}_\star \bigr\rVert^2 + \norm{\vec{y} - \vec{y}_\star}^2  \quad \text{for }\vec{y}_\star = \mat{Q}^\top\vec{b}.
\end{equation*}
Derezi{\'n}ski \& Warmuth \cite[Thm.~8]{derezinski2018reverse} demonstrates that $\vec{y}_0 = \mat{Q}_{\set{S}}^{-1} \vec{b}_{\set{S}}$ satisfies
\begin{equation*}
    \expect\,\norm{\vec{b} - \mat{Q} \vec{y}_0}^2 \leq (r + 1)\, \norm{\vec{b} - \mat{Q} \vec{y}_\star}^2
    \text{ and equivalently }
    \expect\, \norm{ \vec{y}_0 -  \vec{y}_\star }^2
    \leq r\, \norm{\vec{b} - \mat{Q} \vec{y}_\star}^2.
\end{equation*}
Conditional on $\vec{y}_0$, TARK achieves a fast convergence rate
\begin{align*}
    \expect\,\Bigl[\, \bigl\lVert  \vec{b} - \mat{Q} \overline{\vec{y}}_t \bigr\rVert^2 \, \Big| \, \vec{y}_0 \Bigr]
    &=  \bigl\lVert \vec{b} - \mat{Q} \vec{y}_\star \bigr\rVert^2
    + \expect \Bigl[\, \bigl\lVert \overline{\vec{y}}_t - \vec{y}_\star \bigr\rVert^2 \, \Big| \, \vec{y}_0 \Bigr]
    \\
    &\leq 
     \biggl[1 + \frac{2r - 1}{t - t_{\rm b}}\biggr] \,\bigl\lVert  \vec{b} - \mat{Q} \vec{y}_\star \bigr\rVert^2 + 
     \biggl(1 - \frac{1}{r} \biggr)^{t_{\rm b}} \,\norm{\vec{y}_0 - \vec{y}_\star}^2.
\end{align*}
By averaging over $\vec{y}_0$, the overall convergence rate is
\begin{equation*}
    \expect\, \bigl\lVert  \vec{b} - \mat{Q} \overline{\vec{y}}_t \bigr\rVert^2 
    \leq \biggl[1 + \biggl(1 - \frac{1}{r}\biggr)^{t_{\rm b}} r + \frac{2r - 1}{t - t_{\rm b}} \biggr] \,\bigl\lVert \vec{b} -  \mat{Q} \vec{y}_\star \bigr\rVert^2.
\end{equation*}
Since $\bigl\lVert  \vec{b} - \mat{Q} \overline{\vec{y}}_t \bigr\rVert^2 = \bigl\lVert  \vec{b} - \mat{A} \hat{\vec{x}} \bigr\rVert^2$ and $\bigl\lVert \vec{b} -  \mat{Q} \vec{y}_\star \bigr\rVert^2 = \bigl\lVert  \vec{b} - \mat{A} \vec{x}_\star \bigr\rVert^2$, this completes the proof.
\end{proof}

The problem of approximately solving a least-squares problem from a small number of entry evaluations of the vector $\vec{b}$ has also received recent attention in the context of \emph{active learning} \cite{chen2019active,musco2022active}.
Existing approaches achieve the guarantee $\lVert  \vec{b} -  \mat{A} \hat{\vec{x}} \rVert^2 \leq (1 + \varepsilon) \, \lVert  \vec{b} - \mat{A} \vec{x_\star} \rVert^2$ with high probability after accessing just $\mathcal{O}(r/\varepsilon)$ \cite{chen2019active} or $\mathcal{O}(r \log r + r/\varepsilon)$ \cite{woodruff2014sketching,derezinski2018leveraged} entries of $\vec{b}$.
Compared to this previous work, Proposition~\ref{prop:active_learning} attains nearly the optimal rate and is among the simplest and most explicit bounds for active linear regression methods.

\section{Proofs for ridge regression}

\label{app:TARK-RR_proof}
This section proves the \rkrr{} and TA\rkrr{} error bounds. The analysis roughly parallels the analysis in Section~\ref{sec:proof}.
However, the proof of Theorem~\ref{thm:rkrr} requires a new strategy, since there is not a simple one-step recursion bounding $\expect\, \norm{\vec{x}_{t+1} - \vec{x}_\mu}^2$
in terms of $\expect\, \norm{\vec{x}_t - \vec{x}_\mu}^2$.
Instead, it is necessary to use a bias--variance decomposition inspired by \cite{defossez2015averaged, jain2018parallelizing}.

\begin{lemma}[Multi-step expectations]
\label{lem:rkrr_2}
    The RK-RR iteration \eqref{eq:alternating} satisfies
    \begin{equation*}
        \expect\, \bigl[\vec{x}_s - \vec{x}_{\mu} \, \big|\, \vec{x}_r \bigr]
        = \mu^{s-r} \Biggl[\mathbf{I} - \frac{\mat{A}^\top \mat{A}}{\norm{\mat{A}}_{\rm F}^2}\Biggr]^{s-r} \bigl( \vec{x}_r - \vec{x}_{\mu} \bigr),
    \end{equation*}
    for any $r < s$,
    where the expectation averages over the random indices $i_r, \ldots, i_{s-1}$.
\end{lemma}
\begin{proof}
For any $t \geq 0$, rewrite the RK-RR iteration \eqref{eq:alternating} as
\begin{equation} \label{eq:kr-rr-rewritten}
    \vec{x}_{t+1} - \vec{x}_{\mu}
    = \mu\, \biggl[\mathbf{I} - \frac{\vec{a}_{i_t} \vec{a}_{i_t}^\top}{\lVert \vec{a}_{i_t} \rVert^2}\biggr] \bigl(\vec{x}_t - \vec{x}_{\mu}\bigr)
    + \mu\, \frac{b_{i_t} - \vec{a}_{i_t}^\top \vec{x}_{\mu}}{\norm{\vec{a}_{i_t}}^2} \vec{a}_{i_t}
    - (1 - \mu) \vec{x}_{\mu}.
\end{equation}
By averaging over the random index $i_t$,
\begin{equation*}
    \expect\, \bigl[ \vec{x}_{t+1}\, \big|\, \vec{x}_t \bigr]
    = \mu\, \Biggl[\mathbf{I} - \frac{\mat{A}^\top \mat{A}}{\norm{\mat{A}}_{\rm F}^2}\Biggr] \vec{x}_t
    + \mu\, \frac{\vec{A}^\top \bigl(\vec{b} - \mat{A} \vec{x}_\mu\bigr)}{\norm{\mat{A}}_{\rm F}^2}
    - (1 - \mu) \vec{x}_{\mu}.
\end{equation*}
The ridge-regularized solution $\vec{x}_{\mu}$ is characterized by $\vec{A}^\top \bigl(\vec{b} - \mat{A} \vec{x}_\mu\bigr) = \frac{1-\mu}{\mu} \norm{\mat{A}}_{\rm F}^2 \,\vec{x}_\mu$, so the last two terms cancel. 
Hence, by averaging over the random indices $i_r, \ldots, i_{s-1}$,
\begin{equation}
\label{eq:TARK-RR_exp_x}
    \expect\,\bigl[\vec{x}_{t+1}\bigr]
    = \mu\, \Biggl[\mathbf{I} - \frac{\mat{A}^\top \mat{A}}{\norm{\mat{A}}_{\rm F}^2}\Biggr] \expect\,\bigl[\vec{x}_t\bigr]
    \quad \text{for each } t \in \{r, \ldots, s-1\}.
\end{equation}
The result follows by chaining these equations together.
\end{proof}

\begin{lemma}[Demmel condition number bound]
\label{lem:demmel_2}
For any $\vec{x}\in\operatorname{range}(\mat{A}^\top)$ and any $s \geq 0$,
\begin{equation*}
        \vec{x}^\top  
        \Biggl[\mathbf{I} - \frac{\mat{A}^\top \mat{A}}{\norm{\mat{A}}_{\rm F}^2}\Biggr]^s
        \vec{x}
        \leq (1 - \kappa_{\rm dem}^{-2})^s\, \lVert \vec{x} \rVert^2.
    \end{equation*}
\end{lemma}
\begin{proof}
The result follows by expanding $\vec{x}$ in $\mat{A}$'s right singular vectors. 
\end{proof}

\begin{proof}[Proof of Theorem~\ref{thm:rkrr}]
To analyze the RK-RR iteration \eqref{eq:kr-rr-rewritten}, introduce a bias sequence $\vec{m}_t$ and a variance sequence $\vec{v}_t$ that are recursively defined by
\begin{align*}
& \vec{m}_0 = \vec{x}_0 - \vec{x}_\mu,
&& \vec{m}_{t+1} = \mu\, \biggl[\mathbf{I} - \frac{\vec{a}_{i_t}^{\vphantom{\top}} \vec{a}_{i_t}^\top}{\lVert \vec{a}_{i_t} \rVert^2}\biggr] \,\vec{m}_t, \\
& \vec{v}_0 = \vec{0},
&& \vec{v}_{t+1} = \mu\, \biggl[\mathbf{I} - \frac{\vec{a}_{i_t}^{\vphantom{\top}} \vec{a}_{i_t}^\top}{\lVert \vec{a}_{i_t} \rVert^2}\biggr] \,\vec{v}_t + \mu \frac{b_{i_t} - \vec{a}_{i_t}^\top \vec{x}_{\mu}}{\norm{\vec{a}_{i_t}}^2} \vec{a}_{i_t}
    - (1 - \mu) \vec{x}_{\mu}.
\end{align*}
By mathematical induction, the sequences satisfy $\vec{x}_t - \vec{x}_\mu = \vec{m}_t + \vec{v}_t$ for each $t \geq 0$, and also $\vec{m}_t, \vec{v}_t \in  \operatorname{range}(\mat{A}^\top)$ for each $t \geq 0$. Intuitively, $\vec{m}_t$ captures the error due to the initial bias $\vec{x}_0 - \vec{x}_\mu$, and $\vec{v}_t$ captures the remaining error.

Using the bias--variance decomposition, it follows that $\lVert \vec{x}_t - \vec{x}_\mu \rVert^2 \leq 2\lVert \vec{m}_t \rVert^2 + 2 \lVert \vec{v}_t \rVert^2$, and hence
\begin{equation*}
    \expect\, \bigl\lVert \vec{x}_t - \vec{x}_\mu \bigr\rVert^2
    \leq \underbrace{2 \expect\, \norm{\vec{m}_t}^2}_{\text{square bias term}} + \underbrace{2 \expect\, \norm{\vec{v}_t}^2}_{\text{variance term}}.
\end{equation*}
The rest of the proof analyzes the square bias and variance terms separately.

To bound the square bias term, average over the random index $i_t$ and apply Lemma~\ref{lem:demmel_2}:
\begin{align*}
    \expect\,\bigl[\,\norm{\vec{m}_{t+1}}^2\, \big| \, \vec{m}_t \bigr] 
    &= \mu^2\, \expect\,\biggl[\vec{m}_t^\top \biggl(\mathbf{I} - \frac{\vec{a}_{i_t} \vec{a}_{i_t}^\top}{\lVert \vec{a}_{i_t} \rVert^2}\biggr)\, \vec{m}_t \, \biggl|\, \vec{m}_t \biggr] \\
    &
    = \mu^2\, \vec{m}_t^\top \biggl[\mathbf{I} - \frac{\mat{A}^\top \mat{A}}{\norm{\mat{A}}_{\rm F}^2}\biggr]\, \vec{m}_t
    \leq \mu^2\,  (1 - \kappa_{\rm dem}^{-2})\, \norm{\vec{m}_t}^2.
\end{align*}
Therefore, by averaging over the random indices $i_0, \ldots, i_{t-1}$,
\begin{equation*}
    \expect\, \norm{\vec{m}_{t+1}}^2
    \leq \mu^2\,  (1 - \kappa_{\rm dem}^{-2})\, \expect\, \norm{\vec{m}_t}^2,
    \quad \text{for each } t \in \{0, \ldots, r-1\}.
\end{equation*}
This equation implies
\begin{equation*}
    \expect\, \norm{\vec{m}_t}^2 \leq [\mu^{2}\,  (1 - \kappa_{\rm dem}^{-2})]^{\,t}\, \norm{\vec{x}_0 - \vec{x}_\mu}^2,
\end{equation*}
which is an exponentially decreasing bound on the square bias.

The analysis of the variance is more delicate. Since $\vec{v}_t$ follows the same recurrence as $\vec{x}_t$, the relation \eqref{eq:TARK-RR_exp_x} from the proof of Lemma~\ref{lem:rkrr_2} can also be applied to $\vec{v}_t$, yielding
\begin{equation*}
    \expect\bigl[\vec{v}_{t+1}\bigr]
    = \mu\, \Biggl[\mathbf{I} - \frac{\mat{A}^\top \mat{A}}{\norm{\mat{A}}_{\rm F}^2}\Biggr] \expect\,\bigl[\vec{v}_t\bigr]
    \quad \text{for each } t \geq 0.
\end{equation*}
This condition together with the initial condition $\vec{v}_0 = \vec{0}$ shows that $\expect\,[\vec{v}_{t+1}] = \vec{0}$ for each $t\ge 0$,
and consequently
\begin{equation}
\label{eq:use_me}
    \expect\, \norm{\vec{v}_{t+1}}^2
    \leq \expect\, \norm{\,\vec{v}_{t+1}
    + (1-\mu) \vec{x}_{\mu}\,}^2 
    \quad \text{for each } t\geq 0.
\end{equation}

Next, calculate
\begin{align*}
    \norm{\,\vec{v}_{t+1} + (1-\mu) \vec{x}_{\mu}\,}^2
    &= \norm{\,\mu\, \biggl[\mathbf{I} - \frac{\vec{a}_{i_t}^{\vphantom{\top}} \vec{a}_{i_t}^\top}{\lVert \vec{a}_{i_t} \rVert^2}\biggr] \vec{v}_t + \mu \frac{b_{i_t} - \vec{a}_{i_t}^\top \vec{x}_{\mu}}{\norm{\vec{a}_{i_t}}^2} \vec{a}_{i_t}}^2 \\
    &= \mu^2\, \vec{v}_t^\top \biggl[\mathbf{I} - \frac{\vec{a}_{i_t}^{\vphantom{\top}} \vec{a}_{i_t}^\top}{\lVert \vec{a}_{i_t} \rVert^2}\biggr] \,\vec{v}_t + \mu^2\, \frac{\bigl|b_{i_t} - \vec{a}_{i_t}^\top \vec{x}_{\mu}\bigr|^2}{\norm{\vec{a}_{i_t}}^2} \\
    &\leq \mu^2\, \norm{\vec{v}_t}^2 + \mu^2\, \frac{\bigl|b_{i_t} - \vec{a}_{i_t}^\top \vec{x}_{\mu}\bigr|^2}{\norm{\vec{a}_{i_t}}^2}.
\end{align*}
The first line uses the definition of $\vec{v}_{t+1}$, and the second and third lines the fact that $\mathbf{I} - \vec{a}_{i_t}^{\vphantom{\top}} \vec{a}_{i_t}^\top / \lVert \vec{a}_{i_t} \rVert^2$ is an orthogonal projection matrix that is idempotent and annihilates the vector $\vec{a}_{i_t}$.

By averaging over the random index $i_t$, it follows
\begin{equation*}
    \expect\, \bigl[\,\norm{\,\vec{v}_{t+1} + (1-\mu) \vec{x}_{\mu}\,}^2 \, \big| \, \vec{v}_t \bigr]
    \leq \mu^2\, \norm{\vec{v}_t}^2 
    + \mu^2\, \frac{\norm{\vec{b} - \mat{A}^\top \vec{x}_{\mu}}^2}{\norm{\mat{A}}_{\rm F}^2}.
\end{equation*}
Moreover, by averaging over the random indices $i_0, \ldots, i_{t-1}$ and using \eqref{eq:use_me},
\begin{equation*}
    \expect\, \norm{\,\vec{v}_{t+1}\,}^2
    \leq \mu^2\, \expect \norm{\vec{v}_t}^2 
    + \mu^2\, \frac{\norm{\vec{b} - \mat{A}^\top \vec{x}_{\mu}}^2}{\norm{\mat{A}}_{\rm F}^2},
    \quad \text{for each } t \geq 0.
\end{equation*}
This equation leads to a simple bound on the variance
\begin{equation*}
    \expect\, \norm{\vec{v}_t}^2\, \leq \frac{\mu^2}{1 - \mu^2} \,\frac{\norm{\vec{b} - \mat{A}^\top \vec{x}_{\mu}}^2}{\norm{\mat{A}}_{\rm F}^2},
\end{equation*}
which follows because $\sum\nolimits_{s=1}^{\infty} \mu^{2s} = \mu^2 / (1 - \mu^2)$.
The stated result follows from the definition of $\lambda$.
\end{proof}
\begin{proof}[Proof of Theorem~\ref{thm:variance_mu}]
Start by decomposing the mean square error as follows:
\begin{equation*}
    \expect\, \bigl\lVert\, \overline{\vec{x}}_t - \vec{x}_\mu \, \bigr\rVert^2
    = \frac{1}{(t - t_{\rm b})^2}\sum\nolimits_{r,s=t_{\rm b}}^{t-1} \expect\, \bigl[(\vec{x}_r - \vec{x}_\mu)^\top (\vec{x}_s - \vec{x}_\mu) \bigr].
\end{equation*}
Next analyze the terms $\expect\, \bigl[(\vec{x}_r - \vec{x}_\mu)^\top (\vec{x}_s - \vec{x}_\mu) \bigr]$ for $r \leq s$ using Lemmas~\ref{lem:rkrr_2} and \ref{thm:rkrr}:
\begin{align*}
    \expect\, \bigl[(\vec{x}_r - \vec{x}_\mu)^\top (\vec{x}_s - \vec{x}_\mu) \bigr]
    &= \expect\, \bigl[(\vec{x}_r - \vec{x}_\mu)^\top \expect\,\bigl[\vec{x}_s - \vec{x}_\mu \, \big|\,  \vec{x}_r \bigr] \bigr] \\
    &= \mu^{s-r}\, \expect\, \Biggl[\bigl(\vec{x}_r - \vec{x}_\mu \bigr)^\top \Biggl[\mathbf{I} - \frac{\mat{A}^\top \mat{A}}{\norm{\mat{A}}_{\rm F}^2}\Biggr]^{s-r} (\vec{x}_r - \vec{x}_\mu) \Biggr] \\
    &\leq \mu^{s-r}\, \expect \,\lVert \vec{x}_r- \vec{x}_\mu \rVert^2 \\
    &\leq \underbrace{2 \mu^{r + s}\, (1 - \kappa_{\rm dem}^{-2})^r \, \bigl\lVert \vec{x}_0 - \vec{x}_\mu \bigr\rVert^2}_{\text{term A}}
    + \underbrace{\frac{2 \mu^{s-r+1}}{\lambda\,(1 + \mu)} \, \lVert \vec{b} - \mat{A} \vec{x}_\mu \rVert^2}_{\text{term B}}.
\end{align*}
By bounding term A uniformly as 
\begin{equation*}
2 \mu^{r + s}\, (1 - \kappa_{\rm dem}^{-2})^r \, \bigl\lVert \vec{x}_0 - \vec{x}_\mu \bigr\rVert^2 \leq 2 \bigl[\mu^2 (1 - \kappa_{\rm dem}^{-2})\bigr]^{\,t_{\rm b}} \, \bigl\lVert \vec{x}_0 - \vec{x}_\mu \bigr\rVert^2
\end{equation*}
and explicitly averaging over term B, it follows
\begin{align*}
    \expect\, \bigl\lVert\, \overline{\vec{x}}_t - \vec{x}_\star \,\bigr\rVert^2
    &= \frac{1}{(t - t_{\rm b})^2}\sum\nolimits_{r,s=t_{\rm b}}^{t-1} \expect\, \bigl[(\vec{x}_r - \vec{x}_\mu)^\top (\vec{x}_s - \vec{x}_\mu) \bigr] \\
    &\leq  2\, \bigl[\mu^2\, (1 - \kappa_{\rm dem}^{-2})\bigr]^{\,t_{\rm b}}\, \bigl\lVert \vec{x}_0 - \vec{x}_\mu \bigr\rVert^2
    + \frac{2 \mu \, \lVert \vec{b} - \mat{A} \vec{x_\mu} \rVert^2}{\lambda (1 + \mu)(t - t_{\rm b})^2}\sum\nolimits_{r,s=t_{\rm b}}^{t-1} \mu^{\,|s-r|}.
\end{align*}
Last, apply the coarse bound
\begin{equation*}
    \sum\nolimits_{r,s=t_{\rm b}}^{t-1} \mu^{\,|s-r|}
    \leq (t - t_{\rm b}) \biggl[-1 + 2 \sum\nolimits_{s=0}^{\infty} \mu^s\biggr]
    = (t - t_{\rm b}) \frac{1 + \mu}{1 - \mu},
\end{equation*}
which completes the proof.
\end{proof}





\bibliographystyle{elsarticle-num} 
\bibliography{references}



\end{document}